\documentclass[reqno,11pt]{amsart}
\usepackage{amsmath,amssymb}
\usepackage{graphicx}
\usepackage[top=1in,bottom=1in,left=1in,right=1in]{geometry}
\newtheorem{thm}{Theorem}
\newtheorem{cor}[thm]{Corollary}
\newtheorem{prop}{Proposition}[section]
\newtheorem{lem}[prop]{Lemma}

\theoremstyle{definition}
\newtheorem{rem}{Remark}[section]
\numberwithin{equation}{section}

\def\be{\begin{equation}}
\def\ee{\end{equation}}
\def\eps{\varepsilon}

\font\TenEns=msbm10
\font\SevenEns=msbm7
\font\FiveEns=msbm5
\newfam\Ensfam
\def\Ens{\fam\Ensfam\TenEns}
\textfont\Ensfam=\TenEns
\scriptfont\Ensfam=\SevenEns
\scriptscriptfont\Ensfam=\FiveEns
\def\R{{\Ens R}}

\begin{document}

\title[Blow-up of critical norm]
{Optimal condition for blow-up of the critical $L^q$ norm \\ for the semilinear heat equation}

\subjclass[2010]{35K58, 35B40, 35B44, 35B33.}%
\keywords{Semilinear heat equation, critical $L^q$ norm, type~I blow-up, self-similar profile}

\author[Mizoguchi]{Noriko Mizoguchi}%
\address{Department of Mathematics, Tokyo Gakugei University,
Koganei, Tokyo 184-8501, Japan }
\email{mizoguti@u-gakugei.ac.jp}

\author[Souplet]{Philippe Souplet}%
\address{Universit\'e Paris 13, Sorbonne Paris Cit\'e,
CNRS UMR 7539, Laboratoire Analyse, G\'{e}om\'{e}trie et Applications,
93430 Villetaneuse, France}
\email{souplet@math.univ-paris13.fr}

\begin{abstract} We shed light on a long-standing open question 
 for the semilinear heat equation $u_t = \Delta u + |u|^{p-1} u$.
Namely, without any restriction on the exponent $p>1$ nor on the smooth domain~$\Omega$, we prove that the critical $L^q$ norm blows up 
whenever the solution undergoes {\it type~I~blow-up.}
A~similar property is also obtained for the local critical $L^q$ norm near any blow-up point.

In view of recent results of existence of type~II blow-up solutions with bounded critical $L^q$ norm, 
which are counter-examples to the open question, 
our result  seems to be essentially the best possible result in general setting.
This close connection between type I blow-up and critical  $L^q$  norm blow-up appears to be a completely new observation.

Our proof is rather involved and requires the combination of various ingredients. It is based on analysis in similarity variables 
and suitable rescaling arguments, 
combined with {\it backward uniqueness and unique continuation properties} for parabolic equations. 

As a by-product,  we obtain the nonexistence of self-similar profiles in the critical $L^q$ space.
Such properties were up to now only known for $ p \le p_S $ and in radially symmetric case for $ p > p_S $, 
where $ p_S $ is the Sobolev exponent. 
\end{abstract}

\maketitle


\section{Introduction and main results}

\medskip

\subsection{Background}
We consider the semilinear heat equation
\be\label{pbmP}
\left\{
\begin{array}{llll}
\hfill u_t &=&\Delta u+|u|^{p-1}u, &\quad t>0,\ x\in \Omega, \\
\noalign{\vskip 1mm}
\hfill u&=&0, &\quad t>0,\ x\in \partial\Omega, \\
\noalign{\vskip 1mm}
\hfill u&=&u_0(x), &\quad t>0,\ x\in\Omega,
\end{array}
\right.
\ee 
where $p>1$ and  $\Omega$ is a smooth (possibly unbounded) domain of $\R^N$ with $N\ge 1$.
For any $u_0\in L^\infty(\Omega)$,
problem~\eqref{pbmP} has a unique, maximal classical solution $u$.
We denote by $T=T(u_0)\in (0,\infty]$ its maximal existence time.
If $T<\infty$, then blow-up occurs, in the sense that
$$\lim_{t\to T} \|u(t)\|_\infty=\infty.$$
Here and hereafter, we denote by $ \| \cdot \|_q $ the norm of $ L^q (\Omega) $ for $ 1 \le q \le \infty $.
The notation is also extended to $q\in (0,1)$ (although this is not a norm).
In all this article, we use the notation
$$p_S:=(N+2)/(N-2)_+,\qquad \beta=1/(p-1)$$
and
$$q^*=N(p-1)/2$$
(note that we allow values of $q^*<1$).

\medskip

The study of the well-posedness of problem \eqref{pbmP} in Lebesgue spaces goes back to the late 1970's-early 1980's.
As is well known, for nonlinear PDE's in general, such properties are not only interesting in themselves, but also play an important role
in the qualitative analysis of solutions (global existence, blow-up behavior); see, e.g. the monograph \cite{QS07} and 
the references therein.
The space $L^{q^*}$ is invariant under the natural scaling of the equation, given by
$$u(t,x)\mapsto \lambda^{2\beta} u(\lambda^2t,\lambda x),\quad \lambda>0,$$
and it is known that the exponent $q^*$ is critical for local well-posedness in $L^q$ spaces and for the blow-up of $L^q$ norms. Namely,
for $q\ge 1$, we have:

\vskip 2pt
$\bullet$ If $q>q^*$, then problem \eqref{pbmP} with $u_0\in L^q(\Omega)$ is well posed, locally in time, 
and the existence time is uniformly positive for $u_0$ in bounded
sets of $L^q(\Omega)$ (see \cite{W80}, \cite{BC96}). As a consequence, we have
$$ 
T<\infty\quad \Longrightarrow \quad\lim_{t\to T}\|u(t)\|_q=\infty\ \hbox{ for all $q>q^*$.}
$$ 

\vskip 2pt
$\bullet$ If $1\le q<q^*$, then problem \eqref{pbmP} with $u_0\in L^q(\Omega)$ is ill posed (both local existence and local uniqueness may fail; 
see \cite{W80}, \cite{HW82}, \cite{CDNW17})
and moreover (see \cite{FML85}), there exist solutions such that  $T<\infty$ and 
$$\sup_{t\in (0,T)}\|u(t)\|_q<\infty.$$ 

\vskip 2pt
$\bullet$ If $q=q^*>1$, then problem \eqref{pbmP} with $u_0\in L^q(\Omega)$ is still well posed (see \cite{W80}, \cite{BC96}), 
but the existence time is no longer uniformly positive for $u_0$ in bounded sets of $L^q(\Omega)$ (see \cite{BC96}),
so that blow-up of the critical  $L^q$  norm when $T<\infty$, even in the sense
\be\label{BUqsupcrit}
\sup_{t\in(0,T)}\|u(t)\|_{q^*}=\infty,
\ee
cannot be deduced from the local existence theory.

\vskip 1pt

The {\it critical  $L^q$  norm blow-up problem}, i.e.~the question whether \eqref{BUqsupcrit} is true when $T<\infty$,
has attacted a lot of attention in the past decades, but is still not fully understood.
The following sufficient conditions for \eqref{BUqsupcrit} when $T<\infty$ are known:

\begin{itemize}
\vskip 1pt
 \item[(a)] $p=1+\frac{4}{N}$ (cf.~\cite{Bre0} and see also \cite[Proposition~16.3]{QS07});
\vskip 2pt
 \item[(b)]  $p>1+\frac{2}{N}$ and $u_t\ge 0$ (see \cite{W83}, \cite{Mat98});
\vskip 2pt
\item[(c)]  $p=1+\frac{2}{N}$, $u\ge 0$, and $\Omega=\R^N$ or $\Omega$ bounded and $u$ bounded near $\partial\Omega$ (see \cite{QSb2}); 
\vskip 2pt
 \item[(d)] $1<p<p_S$ and $\Omega$ convex (see \cite{GK87});
\vskip 2pt
 \item[(e)]  $p=p_S$, $\Omega=B_R$ or $\R^N$, and $u\ge 0$ radially symmetric (this follows from \cite{FHV}, \cite{MM03});
\vskip 2pt
 \item[(f)]  $p>p_S$, $\Omega=B_R$ or $\R^N$, and $u$ radially symmetric (see \cite{Mat98}, \cite{mm_jfa09}).
\vskip 1pt
\end{itemize}
 Here and hereafter we denote by $ B_r(a) $  the ball  with radius $ r $ centered at $ a $ in $\R^N$, and simply by $ B_r $ in the case
$ a = 0 $. 
Up to recently, no counter-examples to property \eqref{BUqsupcrit} were known,
and it was somehow conjectured that \eqref{BUqsupcrit} should be true for any blow-up solution of problem~\eqref{pbmP}.
However, it turns out that this is {\it not} the case: some special blow-up solutions constructed in \cite{schweyer_jfa12}, \cite{dpmw18} 
do satisfy
\be\label{noBUqcrit}
\sup_{t\in(0,T)}\|u(t)\|_{q^*}<\infty.
\ee
 Since  the $ L^{q^*} $ behavior of the solutions  was not considered there, we give the proof in Appendix. 
This concerns the case $p=p_S$ with $N=4$ or $5$ (see \cite{schweyer_jfa12}, \cite{dpmw18}, respectively)
and includes radial as well as nonradial situations. 
 A key feature of these special solutions is that they undergo {\it type~II blow-up}.
Let us recall that blow-up is said to be of type~I if
\be\label{typeIBU}
\limsup_{t\to T} \ (T-t)^\beta \|u(t)\|_\infty<\infty,
\ee
and type~II otherwise.
Type~I blow-up is consistent with the natural scaling of the equation.
In particular, if \eqref{typeIBU} holds, then the blow-up rate is actually comparable to that of the corresponding ODE, namely
$$C_1\le (T-t)^\beta \|u(t)\|_\infty\le C_2, \ \quad t\in (0,T),$$
for some constants $C_1, C_2>0$. Sufficient conditions ensuring type~I blow-up are summarized in Remark~\ref{Rem1}(i) below.

In view of the above results, a natural question is then:
\be\label{newconj}
\hbox{Does the critical   $L^q$ norm blow up whenever blow-up is type~I ?}
\ee
We will show that \eqref{newconj} is indeed true in full generality, for any $p>1$ and without any restriction
on $\Omega$ or on the sign of $u$.
Moreover, we will show that the same property is true for the {\it local} critical   $L^q$ norm near any blow-up point. 
Recall that $a\in \overline\Omega$ is said to be a blow-up point of $u$ if 
$\limsup_{t\to T,\,x\to a}|u(t,x)|=\infty$.

\medskip

\subsection{Main results}

Our main  theorem  is the following.

\goodbreak

\begin{thm}\label{thm1}
Let $\Omega$ be any smooth domain of $\R^N$, $p>1$ and $u_0\in L^\infty(\Omega)$. 
Let $u$ be a type~I blow-up solution of~\eqref{pbmP}.
 Then the following holds. 
\begin{itemize}
 \item[(i)]  We  have 
 $$\lim_{t\to T}\|u(t)\|_{L^{q*}(\Omega)}=\infty.$$
 \item[(ii)]  For any blow-up point $a\in \overline\Omega$ and any $r>0$, we have
$$\lim_{t\to T}\int_{\Omega\cap B_r(a)} |u(t,x)|^{q^*}dx=\infty.$$

\end{itemize}
\end{thm}

In view of the above mentioned counter-examples, 
Theorem~\ref{thm1} seems to be essentially the best possible result in general setting,
and this close connection between type I blow-up and critical   $L^q$ norm blow-up appears to be a completely new observation.
On the other hand, an important special class of blow-up solutions of \eqref{pbmP} consists of (classical) backward
self-similar solutions.
These are of the form 
\be\label{selfsimsol}
u(t,x)=(T-t)^{-\beta}\varphi\Bigl({x\over\sqrt{T-t}}\Bigl),\ \quad t\in[0,T),\ x\in \R^N,
\ee
where $\varphi\in C^2(\R^N)$ is a nontrivial bounded function, called the (backward self-similar)
profile.
A direct computation shows that $u$ in \eqref{selfsimsol} is a solution of 
 \eqref{pbmP} if and only if $\varphi$ solves the elliptic equation
\be\label{profileeqn}
\Delta\varphi -{y\over 2}\cdot\nabla\varphi+|\varphi|^{p-1}\varphi-\beta\varphi=0,\qquad y\in\R^N.
\ee
Moreover, for any backward self-similar solution $u$, we have $\|u(t)\|_{L^{q*}(\R^N)}=\|\varphi\|_{L^{q*}(\R^N)}\ (\le\infty)$ 
for all $t\in[0,T)$.
As a by-product of Theorem~\ref{thm1}, we thus have the following consequence for all possible backward
self-similar profiles.

\begin{cor}\label{cor2}
 Let $N\ge 1$  and $p>1$. 
Let $\varphi\in C^2(\R^N)$ be a nontrivial, bounded classical solution of \eqref{profileeqn}.
Then 
\be\label{profilenotLq}
\varphi\not\in L^{q*}(\R^N).
\ee
\end{cor}

Let us recall (cf.~\cite{GK85}) that for $p\le p_S$, the only bounded backward
 self-similar profiles are the constants $\varphi=\pm \kappa$ 
with $\kappa=\beta^\beta$
 (which implies conclusion \eqref{profilenotLq} in that case).
For $N\ge 3$ and 
$$p_S<p<p_L:=1+\frac6{(N-10)_+},$$ there exist nonconstant, radial bounded profiles
(see \cite{Lep88}, \cite{BQ89}, \cite{Lep90}),
and they are known to behave like $r^{-2/(p-1)}$ as $r\to\infty$, so that \eqref{profilenotLq} is in particular  true. 
As for nonradial profiles, they seem quite difficult to study for $p>p_S$. 
The conclusion \eqref{profilenotLq} is therefore significant.

\medskip

\begin{rem}  \label{Rem1}
(i) {\bf Type~I blow-up.} Blow-up is known to be type~I whenever $1<p<p_S$ and $\Omega$ is convex
(see \cite{GK87}, \cite{GMS03}, \cite{GMS04b}),
or $p>1$ and $u_t\ge 0$ (cf.~\cite{FML85}, \cite{QSb2}). 
For positive solutions, this remains true for any $\Omega$ (possibly nonconvex) if either $N=2$ or $p<N(N+2)/(N-1)^2$ $(<p_S)$;
see \cite{PQS2}, \cite{q_ma16}. However it is still unclear if this is true in the full range $(1,p_S)$ for nonconvex domains.

Concerning the supercritical range  in the Sobolev sense,  it is known (see \cite{MM03}, \cite{MizJDE}) that for radial solutions, blow-up is always type~I 
in a ball and under some conditions in $\R^N$ for $p_S<p<p_{JL}$, 
where 
$$p_{JL}:=1+4\frac{N-4+2\sqrt{N-1}}{(N-2)(N-10)_+}.$$ 
In the supercritical and critical range  in the Sobolev sense,  type I blow-up is also known to occur for some solutions which 
are neither radial nor increasing in time:
see \cite{crs_mams17pr} for $N=3$, $p>p_S$,
\cite{mrs17} for $N=4$, $p>5$,
\cite{cmr_cmp17} for $N\geq7$ and $p=p_S$
and \cite{cmr_cras17} for $p=p_S$.

\smallskip

(ii) {\bf Type~II blow-up.} The existence of type~II solutions is up to now only known to occur for certain values in the range $p\ge p_S$:

\vskip 2pt
$\bullet$ $p=p_S$ (see \cite{schweyer_jfa12}, \cite{dpmw18});
\vskip 1pt

$\bullet$ 
$p=(N+1)/(N-3)$ with $N\ge 7$ (see \cite{dpmw17a}); 
\vskip 1pt

$\bullet$ $p\ge p_{JL}$; see \cite{HV94a}, \cite{HV94}, \cite{MizADE}, \cite{Sek18} in radial case and \cite{collot_apde17}, \cite{cmr17} 
in nonradial case.
\vskip 2pt

We note that one cannot expect general conclusions regarding the blow-up of the critical   $L^q$ norm
for type~II blow-up solutions, since there are examples of boundedness of the critical   $L^q$ norm 
as stated above (cf. \cite{schweyer_jfa12}, \cite{dpmw18})
as well as of unboundedness.

In fact, for $p>p_S$, radial type~II blow-up solutions converge to the singular stationary solution $ c_* |x|^{-2/(p-1)} $ as $ t $ tends to 
$T$ (see \cite{mm_jfa09}), and this implies unboundedness of $ L^{q^*} $ norm.
As for the nonradial solutions in \cite{collot_apde17}, \cite{cmr17},  \cite{dpmw17a}, the unboundedness of $L^{q^*}$ norm
 can be directly seen by simple computations, using the asymptotic form of the solution obtained in those works.
 
On the other hand, it seems a difficult open problem whether or not boundedness of the critical   $L^q$ norm may happen for 
 type~II blow-up solutions with $p\ne p_S$.
 In fact, the construction of the known nonradial type II blow-up solutions is somewhat related to radial type II blow-up profiles 
in lower dimensions and we
have no information on nonradial type II blow-up solutions completely independent of radial type~II profiles. 
\end{rem}

\begin{rem}  \label{Rem2}
(i)  The results (b)-(e) mentioned in section~1.1 above 
become special cases of Theorem~\ref{thm1}, since blow-up is known to be type~I under the corresponding assumptions.
On the other hand, for any $p \le p_S $, nothing seems to be known on blow-up type for nonmonotone sign-changing solutions 
when $\Omega$ is nonconvex.
Therefore, the result in (a) in the very special case $p=1+4/N$ 
is not a consequence of Theorem~\ref{thm1}.

\smallskip

(ii) In Theorem~\ref{thm1}(i), the function $\|u(t)\|_{L^{q*}(\Omega)}$ is allowed to take the value $\infty$ 
(which arises, when $\Omega$ is unbounded, if $u(t)\not\in L^{q*}(\Omega)$ at some $t$).
Of course, in the case $q^*\ge 1$ (i.e. $p\ge 1+2/N$), if we assume $u_0\in L^{q*}(\Omega)$, then $u(t)\in L^{q*}(\Omega)$ for all $t\in (0,T)$.

\smallskip
(iii) As a consequence of the proof of Theorem~\ref{thm1}, for any $p>1$ and any type~I blow-up solution, we obtain an $L^{q*}$-concentration property 
in backward space-time parabolas near every blow-up point (see Lemma~\ref{lem31}), which may be of independent interest.
\end{rem}

\subsection{Outline of proof}

The proof of Theorem~\ref{thm1} is rather involved and requires the combination of various ingredients.
Let us sketch the proof.
We assume for contradiction that there exists a solution $ u $ undergoing type~I blow-up at $ t = T $, whose $ L^{q*} $ norm stays bounded
along a sequence of times $t_n\to T$.

We first show a concentration property for the $ L^{q*} $ norm near blow-up points,
by means of an analysis in backward similarity variables. 
It is classical to analyze global solution of the new equation generated by the change of variables, instead of the original solution 
blowing up in finite time (cf. \cite{GK85}, \cite{GK87}).
The fact that the energy is nonincreasing in time in convex domains plays a crucial role there.
We here deal with all domains not necessarily convex. 
One of the main points in our method is to use a criterion to exclude blow-up at a given point
involving a certain weighted $L^1$ norm, based on suitable delayed smoothing effects in weighted $L^q$ spaces.

Next supposing without loss of generality that the origin is a blow-up point,
we introduce the sequence of rescaled solutions 
$$ v_n (s,y) = (T-t_n)^\beta u\bigl( T+(T-t_n)s, \sqrt{T-t_n} \,y\bigr).$$
Using boundedness of the (local) $L^{q*}$ norm, type~I estimates and the above delayed smoothing effects in similarity variables,
we can show that, up to a subsequence:
$$v(s,y) := \lim_{n \to \infty} v_n (s,y) \quad\hbox{exists and solves  \eqref{pbmP} in $(-2,0)\times D$}$$
(where $D$ is the whole space or a half-space depending on the location of the blow-up point),
$$ |v_n (s,y)| \leq C,\quad\hbox{ for $s \in [-1,0)$ and $|y| \geq R$, \ \ with some $ C, R > 0,$}$$ 
and
$$\int_{B_1(y_0)}  |v_n(0,y)|^{q^*} dy \to 0 \quad\hbox{as $ n \to \infty$,\ \ for $|y_0| \ge R$.}$$
From this it can be deduced 
that $v$ is bounded in $[-1,0)\times [\{|y| \ge R\} \cap D]$ and that
$$ v(0,y) = 0 \quad\hbox{ in $\{|y| \ge R\} \cap D.$}$$
Applying backward uniqueness in a half-space (cf.~\cite{Escauriaza-Seregin-Sverak:UMN58}) and then unique continuation, 
we infer that $v\equiv 0$ in $[-1,0)\times D$.
This leads to a contradiction with the concentration property given in the first part.
\smallskip

\begin{rem}  \label{Rem3}
(i) Let us mention that a method based on backward uniqueness and unique continuation was  
used in \cite{Escauriaza-Seregin-Sverak:UMN58} (see also references therein) to show that
if a solution to the Navier-Stokes equations in $ \R^3 $ has a singularity,
then the $ L^3 $ norm diverges to $ \infty $.
The $ L^3 $ norm is critical there, however the situation is quite different from ours.
In particular, while type~I estimate is not needed in \cite{Escauriaza-Seregin-Sverak:UMN58}, it is crucially needed here 
since there exist counter-examples to the conclusion of Theorem~\ref{thm1} without type~I estimate. 
 At the level of our proof, the type~I estimate
 is essential in order to show the $ \varepsilon $-regularity and the concentration of the local critical $L^q$ norm.
The type~I estimate is not needed in the arguments of \cite{Escauriaza-Seregin-Sverak:UMN58},
thanks to the well-known {\it energy inequality} (see e.g. Definition 2.1 of \cite{Escauriaza-Seregin-Sverak:UMN58}),
whose form is different from the energy relation for \eqref{pbmP}, owing to the divergence structure of the Navier-Stokes equations.

On the other hand, a related method was recently applied in \cite{MizKS} to prove that only type~II blow-up is possible 
in the parabolic-parabolic Keller-Segel system in $\R^2$
(unlike in higher dimensions $N\ge 3$, where type~I blow-up does occur).
The possibility of type~I blow-up was ruled out there by a contradiction argument. 
The system consists of two parabolic equations for two components, in which one equation is linear and the other has nonlinear drift 
term, and the
 linear equation is helpful in a priori estimates.
 
More generally, it is common in the case of the Keller-Segel system and 
nonlinear heat equation 
to need type~I estimate (see, e.g., the monograph \cite{QS07}).
The methods to obtain the so called $\eps$-regularity and to show concentration of the local critical  $L^q$ norm 
at a blow-up point are essentially different in Navier-Stokes, Keller-Segel and nonlinear heat equations, due to  
distinctive structural features.

\smallskip

(ii) We point out that for problem \eqref{pbmP} in the special case $p<p_S$ and $\Omega$ convex,
blow-up of the critical  $L^q$ norm was obtained in \cite{GK87}
as a consequence of the analysis of the {\it local self-similar blow-up profile}.
The local self-similar profile analysis was extended in \cite{im_die03} to the case $p\le p_S$, $\Omega$ bounded and $u_0\ge 0$
under the assumption of type~I blow-up, and one can thereby infer blow-up of the critical  $L^q$ norm in that case.
This approach is restricted to the range $p\le p_S$ and is very different from our proof, which does not require any information on the local self-similar blow-up profile,
and on the contrary {\it provides} new information on this profile as a by-product when $p>p_S$ (cf.~Corollary~\ref{cor2}). 
\end{rem}

\medskip

The outline of the paper is as follows. Section 2 provides some important preliminaries to our proofs
and consists of two subsections.
In subsection~2.1 we recall the classical framework of similarity variables from \cite{GK85} and
we give a suitable version of a criterion from \cite{AHV97},~\cite{Sou09} to exclude blow-up at a given point.
In subsection~2.2 we recall some known 
results on backward uniqueness 
for parabolic equations. 
Section 3 is then devoted to the proof of Theorem~\ref{thm1}.
Finally, in Appendix, we 
verify that the special solutions constructed in \cite{schweyer_jfa12}, \cite{dpmw18}  
remain bounded in $L^{q^*}$ (these papers did not address the $L^{q^*}$ behavior of the solutions).

\section{Preliminaries}

\subsection{Similarity variables and local criterion for excluding blow-up}

Denote by $\tilde u(t,\cdot)$ the extension of $u(t,\cdot)$ by $0$ to the whole of $\R^N$
and let $a\in \R^N$.  
Following~\cite{GK85}, we define the (backward) similarity variables around $(T,a)$ by
$$ s=-\log(T-t),\qquad y={x-a\over \sqrt{T-t}}=e^{s/2}(x-a),$$
set $s_0:=-\log T$, and define the rescaled function by
$$w(s,y)=w_a(s,y):=(T-t)^\beta \tilde u(t,x)=(T-t)^\beta \tilde u\bigl(t,a+y\sqrt{T-t}\bigr),
\ \quad  s_0\le s<\infty,\ y\in\R^N.$$
We note that the rescaling point $a$ is usually taken in $\overline\Omega$.
However, the above is well defined for any $a\in \R^N$ and this will be convenient in what follows.

Denoting the transformed domains by
$$\Omega_s : = e^{s/2} (\Omega-a),\quad s>s_0,$$
the PDE in \eqref{pbmP} becomes, in similarity variables:  
\be\label{newPDE}
w_s-\mathcal{L}w=w^p-\beta w, \qquad  s_0<s<\infty,\ y\in \Omega_s,
\ee
where 
$$\mathcal{L}=\Delta-{y\over 2}\cdot\nabla=\rho^{-1}\nabla\cdot(\rho\nabla),\qquad \rho(y)=e^{-|y|^2/4}.$$
We stress that the PDE in \eqref{newPDE} is of course only satisfied in $\Omega_s$ (not in the whole $\R^N$), and that $w(s,\cdot)=0$ outside $\Omega_s$.
For any $\phi\in L^\infty(\R^N)$, we put
$$\|\phi\|_{L^q_\rho}=\Bigl(\int_{\R^N}|\phi(y)|^q\rho(y)\, dy\Bigr)^{1/q},\quad 1\leq q<\infty,$$
and note that
\be\label{Holderrho}
\|\cdot\|_{L^q_\rho}\leq  C(N)\|\cdot\|_{L^r_\rho},\quad 1\leq q<r<\infty. 
\ee 
\smallskip

The following Proposition~\ref{prop21} is twofold:
\begin{itemize}
\item[-]Assertion (i) gives a local criterion to exclude blow-up at a given point,
which is essentially due to \cite{AHV97},~\cite{Sou09}. 
We here state and prove a version adapted to our needs.
We note that, although the criterion concerns possible blow-up points $a\in\overline\Omega$,
the corresponding estimate \eqref{Conclwzsmall} remains true for any $a\in\R^N$, 
which will be convenient in the proof of Theorem~\ref{thm1}. 
As a key difference with the earlier local criterion in \cite{GK89} and its proof (see also \cite{im_die03}),
the method from \cite{AHV97},~\cite{Sou09}, that we use here, 
relies on delayed smoothing effects in weighted $L^q$ spaces (cf.~Lemma~\ref{lem22}) instead of rescaled energy,
hence allowing to cover possibly nonconvex domains and/or $p>p_S$.
\smallskip

\item[-] Assertion (ii) rules out blow-up at space infinity in case $\Omega$ is unbounded and $u_0$ belongs to $L^m$ for some finite $m$.
When $\Omega$ is convex the conclusion was first proved in \cite{GK89} under the assumption $u_0\in H^1(\Omega)$.
The general case was obtained in \cite{fi_aihp14}.
We shall give a different proof, allowing a common treatment to both assertions (i) and (ii).
\end{itemize}

\begin{prop}\label{prop21}
Let $\Omega$ be any smooth domain of $\R^N$, $p>1$ and $u_0\in L^\infty(\Omega)$. 
Assume that $u$ is a blow-up solution of \eqref{pbmP} satisfying the type~I estimate 
\be\label{typeI}
\|u(t)\|_\infty\le M(T-t)^{-\beta},\quad 0\le t<T,
\ee 
for some constant $M>0$.
\begin{itemize}
 \item[(i)]  Let $a\in \R^N$ and let $w_a$ be the
rescaled solution by similarity variables around $(T,a)$.
For each $q\ge p$ with $q>Np/2$, there exist $s^*,\eps_0,C_0>0$ depending only on $N,p,q,M$ such that, if
\be\label{Hypwzsmall}
 \|w_a(\sigma)\|_{L^1_\rho}<\eps_0\quad\hbox{ for some $\sigma\geq s_0$,}
\ee 
then
\be\label{Conclwzsmall}
\|w_a(\sigma+s)\|_{L^q_\rho}<C_0 e^{-\beta s}  \|w_a(\sigma)\|_{L^1_\rho}\quad\hbox{ for all $s\ge s^*$.}
\ee
Moreover, if $a\in\overline\Omega$ and \eqref{Hypwzsmall} is satisfied, then $a$ is not a blow-up point of $u$.
\smallskip
 \item[(ii)]  Assume that $\Omega$ is unbounded and that $u_0\in L^m(\Omega)$ for some $m\in (0,\infty)$.
Then $u$ does not blow up at infinity, namely there exist $R,C>0$ such that
$$|u(t,x)|\le C,\quad 0<t<T,\ x\in\overline\Omega\cap \{|x|\ge R\}.$$
\end{itemize}
\end{prop}

In view of the proof of Proposition~\ref{prop21}, we introduce the semigroup $(T(s))_{s\geq 0}$ associated with~$\mathcal{L}$. More precisely, 
for each $\phi\in L^\infty(\R^N)$, we set $T(s)\phi:=z(s,\cdot)$,
where $z$ is  the unique classical solution of
$$\left\{ 
\begin{array}{lcll}
\hfill z_s&=&\mathcal{L}z, &\quad s>0,\ y\in \R^N,\\
\noalign{\vskip 1mm}
\hfill z(0,y)&=&\phi(y), &\quad y\in \R^N.
\end{array}
\right.$$
An essential property of $(T(s))_{s\geq 0}$ is given by the following Lemma
(see \cite{Gr76} and cf.~also \cite{HV93}).

\begin{lem}\label{lem22}
\begin{itemize}
 \item[(i)]  For all $1\leq q\leq\infty$,
\be\label{Tcontraction}
\|T(s)\phi\|_{L^q_\rho}\leq \|\phi\|_{L^q_\rho},\qquad s\geq 0,\quad \phi\in L^\infty(\R^N).
\ee 
\vskip 2mm
 \item[(ii)]   {\it (Delayed regularizing effect)} For all $1\leq m<q<\infty$, 
there exist $c_0,s^*>0$ such that
\be\label{delayedregul}
\|T(s)\phi\|_{L^q_\rho}\leq c_0\|\phi\|_{L^m_\rho},\qquad s\geq s^*,\quad \phi\in L^\infty(\R^N).
\ee 
\end{itemize}
\end{lem}

\medskip

\begin{proof}[Proof of Proposition~\ref{prop21}]  
(i) Set $w:=w_a$. For given $s_1\geq s_0$, denote by $W$ the solution of
$$\left\{ \begin{array}{lcll}
\hfill  W_s-\mathcal{L} W
&=&| w |^p-\beta  W, 
&\quad s>s_1,\ y\in\R^N,  \hfill\\
\noalign{\vskip 2mm}
\hfill  W(s_1,y)
&=&|w(s_1,y)|, 
&\quad y\in\R^N.  \hfill
\end{array}\right.$$
Clearly $W$ exists globally 
and we claim that
\be\label{claimComp}
|w|\leq W \quad\hbox{ in $(s_1,\infty)\times \R^N$}.
\ee
To check this, in order to avoid any difficulty related to the application of the maximum principle in a time dependent domain,
we just note that, going back to the original variables by setting 
$$V(t,x):=(T-t)^{-\beta}W\bigl(-\log(T-t),(x-a)/\sqrt{T-t}\bigr),\quad t_1<t<T,\ x\in\R^N,$$
with $t_1=T-e^{-s_1}$, we have
$$\left\{ \begin{array}{lcll}
\hfill  V_t-\Delta V&=&|u|^p,&\quad t_1<t<T,\ x\in\Omega,  \hfill\\
\noalign{\vskip 2mm}
\hfill  V&\ge&0,&\quad t_1<t<T,\ x\in\partial\Omega,  \hfill\\
\noalign{\vskip 2mm}
\hfill  V(t_1,x)&=&|u(t_1,x)|, &\quad x\in\Omega. \hfill
\end{array}\right.$$
Therefore, $u\le V$ and $-u\le V$ in $(t_1,T)\times\Omega$ by the maximum principle.
Consequently, $|\tilde u|\le V$ in $(t_1,T)\times\R^N$, and claim \eqref{claimComp} follows.

Next, by the variation of constants formula, we deduce that
\be\label{varconstphi}
e^{\beta s}| w(s_1+s)|\leq
e^{\beta s} W(s_1+s)\leq T(s)| w (s_1)|
+\int_0^s e^{\beta\tau}
T(s-\tau)| w(s_1+\tau)|^p\,d\tau,\quad
s_1\geq s_0,\ s>0.
\ee 
On the other hand, since 
\be\label{typeIw}
\|w(s)\|_\infty\leq M,\qquad s\geq s_0,
\ee
due to \eqref{typeI}, $W$ satisfies
$$\left\{ \begin{array}{lcll}
\hfill  W_s-\mathcal{L} W
&\le&C_1W, \hfill
&\quad s>s_1,\ y\in\R^N,  \hfill\\
\noalign{\vskip 2mm}
\hfill  W(s_1,y)
&=&|w(s_1,y)|, 
&\quad y\in\R^N,  \hfill
\end{array}\right.$$
where $C_1:=M^{p-1}$.  Therefore
\be\label{expgrowthphi}
|w(s_1+s)|\leq e^{C_1s}T(s)|w(s_1)|,
\qquad s_1\geq s_0,\ s>0.
\ee 
Let $s^*$ be given by Lemma~\ref{lem22}(ii)
with $m=1$ and set $\eps:= \|w(\sigma)\|_{L^1_\rho}$.
It follows from \eqref{Hypwzsmall}, \eqref{expgrowthphi},
and Lemma~\ref{lem22}(i) that 
\be\label{controlphisstar}
\|w(\sigma+s)\|_{L^1_\rho}\leq C_2\eps,\quad 0<s\leq s^*,
\qquad\hbox{ with $C_2:=e^{C_1s^*}$.}
\ee 
Let now
\be\label{defTzero}
T_0=\sup\Bigl\{s>0;\ 
e^{\beta\tau}\|w(\sigma+s^*+\tau)\|_{L^1_\rho}
\leq 2C_2\eps,\ \tau\in [0,s]\Bigr\}.
\ee
Note that $T_0>0$ and suppose for contradiction that $T_0<\infty$. 
We have, by \eqref{controlphisstar} and \eqref{defTzero},
\be\label{expdecayphi}
\|w(\sigma+s^*+s)\|_{L^1_\rho}
\leq 2C_2\eps e^{-\beta s}, 
\qquad -s^*\leq s\leq T_0.
\ee 
For $0\leq \tau\leq T_0$, \eqref{delayedregul}, \eqref{expgrowthphi} and \eqref{expdecayphi}
imply
\begin{align}
\|w(\sigma+s^*+\tau)\|_{L^q_\rho}
&\leq e^{C_1s^*}\|T(s^*)w(\sigma+\tau)\|_{L^q_\rho}\leq c_0e^{C_1s^*}\|w(\sigma+\tau)\|_{L^1_\rho} \notag \\
&\leq 2C_2c_0e^{C_1s^*}\eps e^{-\beta(\tau-s^*)}=C_3\eps e^{-\beta\tau}, \label{expdecayphiB}
\end{align}
with $C_3=2C_2c_0e^{(C_1+\beta)s^*}$.
Using \eqref{Holderrho},  \eqref{Tcontraction}, \eqref{varconstphi}
with $s_1=\sigma+s^*$,  \eqref{controlphisstar} and \eqref{expdecayphiB},
we deduce that, for all $s\in [0,T_0]$,
\begin{align}
e^{\beta s}\|w(\sigma+s^*+s)\|_{L^1_\rho}
&\leq\|w(\sigma+s^*)\|_{L^1_\rho}
+\int_0^s e^{\beta\tau}
\|w(\sigma+s^*+\tau)\|_{L^p_\rho}^p\,d\tau \notag \\
&\leq C_2\eps+ C(N,p)(C_3\eps)^p \int_0^s e^{-\beta(p-1)\tau}\,d\tau \notag \\
&\leq C_2\eps+ C(N,p) C_3^p (\beta(p-1))^{-1}\eps^p. \label{expdecayAa}
\end{align}
Applying this with $s=T_0$, it follows that
$2C_2\eps\leq C_2\eps+ C(N,p) C_3^p(\beta(p-1))^{-1}\eps^p,$
hence $C_2\leq C(N,p,q,M)\eps^{p-1}$.
Assuming $\eps \le \eps_0$ with $\eps_0=\eps_0(N,p,q,M)>0$ sufficiently small,
we thus necessarily have $T_0=\infty$. It then follows from \eqref{expdecayphiB} that
\be\label{expdecayphiC}
\|w(\sigma+s)\|_{L^q_\rho}\leq C_0\eps e^{-\beta s},
\quad s\geq s^*,
\ee 
with $C_0 = C_0(N,p,q,M):=C_3  e^{\beta s^*}=2c_0\exp[2(M^{p-1}+\beta)s^*]$  that is, \eqref{Conclwzsmall} holds. 
(For the proof that $a$ is not a blow-up point, see at the end of the proof of assertion (ii).)
\smallskip

(ii) We may assume without loss of generality that $ m > 1 $.
Fix $q\ge p$ with $q>Np/2$. Pick $A>0$ large enough so that 
$T^{\beta-N/2}\|u_0\|_\infty \int_{|z|>A} \,e^{-|z|^2/4T}\, dz\le \eps_0/2$,
where $\eps_0=\eps_0(N,p,q,M)$ 
 is given by assertion (i).
Since $u_0\in L^m(\Omega)$ with $ m > 1 $, choosing $R_0>2A$ sufficiently large, 
for any $a\in\overline\Omega$ such that $|a|>R_0$, we get
$$
\begin{aligned}
\|w_a(s_0)\|_{L^1_\rho}
&=T^{\beta-N/2} \int_{\R^N} |\tilde u_0(a+z)|\,e^{-|z|^2/4T}\, dz\\
&\le T^{\beta-N/2} \int_{|z|\le A} |\tilde u_0(a+z)|\,e^{-|z|^2/4T}\, dz+T^{\beta-N/2}\|u_0\|_\infty \int_{|z|>A} \,e^{-|z|^2/4T}\, dz\\
&\le T^{\beta-N/2} \Bigl(\int_{|z|\le A} |\tilde u_0(a+z)|^m\, dz \Bigr)^{1/m} 
\Bigl(\int_{|z|\le A}\,e^{-m|z|^2/4(m-1)T}\, dz\Bigr)^{(m-1)/m}+{\eps_0\over 2}\\
&\le T^{\beta-N/2} \Bigl(\int_{|y|\ge R_0/2} |\tilde u_0(y)|^m\, dy \Bigr)^{1/m} 
\Bigl(\int_{\R^N}  e^{-m|z|^2/4(m-1)T}\, dz \Bigr)^{(m-1)/m}+{\eps_0\over 2}<\eps_0. 
\end{aligned}
$$
It then follows from \eqref{Conclwzsmall}  
that
\be\label{expdecayAb}
 \int_{\R^N} |w_a(s_0+s,y)|^q\rho(y)\,dy\le C_0^q\eps_0^q e^{-\beta qs},\quad s\ge s^*.
\ee 

Restated in terms of $u$, \eqref{expdecayAb} means that
$$\int _{\R^N}|\tilde u( t, a+y\sqrt{T-t})|^q e^{-|y|^2/4}\,dy\leq C:=(C_0\eps_0T^{-\beta})^q,\qquad \tau<t<T,\quad |a|\ge R_0,$$
with $\tau:=(1-e^{-s^*})T$. Set $\tau_1:= \max(\tau,T-1)$. For each $a\in\overline\Omega$ with $|a|\ge R:=R_0+2$, 
integrating with respect to $\xi\in B_2(a)$ and using Fubini's theorem and the change of variable $z=\xi+y\sqrt{T-t}$, it follows that,
for all $t\in [\tau_1,T)$,
$$
\begin{aligned}
Ce^{1/4}|B_2(0)|
&\ge \int _{|y|<1}\int_{|\xi-a|<2}|\tilde u(t,\xi+y\sqrt{T-t})|^q\,d\xi\,dy \\
&\ge \int _{|y|<1}\int_{|z-y\sqrt{T-t}-a|<2}|\tilde u(t,z)|^q\,dz\,dy\ge
|B_1(0)| \int_{|z-a|<1}|\tilde u(t,z)|^q\,dz.
\end{aligned}
$$
Since $q>Np/2$, we deduce from standard parabolic regularity properties that 
$$|u(t,x)|\le C_1,\quad \tau_2<t<T,\ x\in\overline\Omega\cap \{|x|\ge R\},$$
for some constant $C_1>0$ and $\tau_2=(\tau_1+T)/2$, which implies assertion (ii).

\smallskip

Finally, to show that $a \in \overline\Omega$ is not a blow-up point under assumption \eqref{Hypwzsmall},
we observe that, by continuity, we have 
$\|w_b(\sigma)\|_{L^1_\rho}<\eps_0$ for all $b\in \R^N$ such that $|b-a|\le r$,
with $r>0$ sufficiently small.
The first part of assertion (i) then guarantees that for $ q > Np/2 $ 
$$ \|w_b(\sigma+s)\|_{L^q_\rho}<C_0\eps_0 e^{-\beta s},\quad s\ge s^*,$$
for all such $b$. The conclusion then follows from a similar argument as in the preceding paragraph.
\end{proof}

\subsection{Backward uniqueness}

One of the key ingredients in our proof of Theorem~\ref{thm1} is a backward uniqueness theorem 
in a half-space, which was given by \cite{Escauriaza-Seregin-Sverak:UMN58} (also see \cite{Escauriaza-Seregin-Sverak:ZNSPOMI288}, 
\cite{Escauriaza-Seregin-Sverak:ARMA}, \cite{Escauriaza-Seregin-Sverak:AA1})
to treat a problem on regularity of solutions to the Navier-Stokes equations. 
Since we are concerned with blow-up problem, the initial data in backward uniqueness theorem has singularities.
However we will deal with an auxiliary problem whose singularities remain confined in a ball,
so that this backward uniqueness theorem applies conveniently to our purpose
(working in half-spaces which do not intersect this ball).
Although the backward uniqueness theorem was given in more general form for example in 
\cite[Theorem 5.1]{Escauriaza-Seregin-Sverak:UMN58}, we describe it in the following form as needed here.
 Note that, as customary, we state it in terms of the backward heat operator.

\begin{prop}
\label{prop23backward-unique}
Let $ \R^N_+ = \{ y = (y_1, y_2, \cdots, y_N) \in \R^N: y_1 > 0 \} $ and $ S > 0 $.
Let $w \in  C^{1,2}([0,S]\times\overline{\R_+^N};\R)$ satisfy
\[
|w (s,y) | \leq e^{ C_1 |y|^2 } \quad \mbox{ in } (0,S)\times  \R_+^N
\]
for some $ C_1 > 0 $.
If $ w $ fulfills
\[
| w_s + \Delta w | \leq C_2 |w| \quad \mbox{ in }  (0,S)\times  \R_+^N
\]
for some $ C_2 > 0 $ and 
\[
w (0,y) = 0 \quad \mbox{ in } \R_+^N,
\]
then $ w\equiv 0$ in $[0,S)\times  \R_+^N$.              
\end{prop}

\section{Proof of Theorem~\ref{thm1}}

As a first consequence of Proposition~\ref{prop21}(i), we have the following lemma, which shows that the $L^{q^*}$ norm 
 of any type~I blow-up solution has to concentrate in space-time parabolas near every blow-up point.

\begin{lem}\label{lem31}
Let $\Omega$ be any smooth domain of $\R^N$, $p>1$ and $u_0\in L^\infty(\Omega)$. 
Assume that $u$ is a blow-up solution of \eqref{pbmP} satisfying the type~I estimate \eqref{typeI} with $M>0$.
Then there exist $\eta_1,  k_1 >0$, 
depending only on $M, p,  N,$ such that for any blow-up point $a\in \overline\Omega$, we have
$$
\int_{\Omega\cap\{|x-a|\le   k_1 \sqrt{T-t}\}} |u(t,x)|^{q^*}\, dx \ge \eta_1,\quad 0\le t<T.
$$
\end{lem}

\begin{proof}[Proof of Lemma~\ref{lem31}]
Let $a\in\overline\Omega$.  Recalling the notation in Section 2.1, we set 
$$w(s,y)=w_a(s,y)=(T-t)^\beta \tilde u(t,a+y\sqrt{T-t}),\qquad s_0<s<\infty,\ y\in \R^N,$$
where $\tilde u$ is the extension of $u$ by $0$ to the whole of $\R^N$ and $s_0=-\log T$.
By Proposition~\ref{prop21}(i), there exists $\eps_0=\eps_0(M,p,N)>0$ such that:
$$\quad\hbox{ if $\|  w(s)\|_{L^1_\rho}\le \eps_0$ for some $s\in [-\log T,\infty)$, then $a$ is not a blow-up point.}$$

Now choose $k_1>0$ such that $\int_{|y|> k_1} e^{-|y|^2/4}\,dy\le \eps_0/2M$. 
If $a$ is a blow-up point then,  using \eqref{typeIw}, we have for all $s \geq s_0$, 
$$\begin{aligned}
\eps_0
&\le \int_{\R^N} |w(s,y)|e^{-|y|^2/4}\,dy \\
&\le \int_{|y|\le k_1} |w(s,y)|e^{-|y|^2/4}\,dy+M\int_{|y|> k_1} e^{-|y|^2/4}\,dy
\le \int_{|y|\le k_1} |w(s,y)|\,dy+\eps_0/2,
\end{aligned}$$
hence
\be\label{proofLemA}
 \int_{|y|\le k_1} |w(s,y)|\,dy\ge \eps_0/2.
\ee 
Moreover, since $q^*=N(p-1)/2$, we have
$$
\int_{|x-a|\le k_1\sqrt{T-t}} |\tilde u(t,x)|^{q^*}\, dx
=(T-t)^{N/2}\int_{ |y|\le k_1} |\tilde u(t,a+y\sqrt{T-t})|^{q^*}\, dy 
=\int_{|y|\le k_1} |w(s,y)|^{q^*}\, dy. 
$$
If $p\ge 1+2/N$, hence $q^*\ge 1$, the conclusion then directly follows from \eqref{proofLemA} and H\"older's inequality.
If $p\in (1,1+2/N)$, hence $q^*\in (0,1)$, the conclusion follows from \eqref{typeIw} and \eqref{proofLemA}, by writing
$$\int_{|y|\le k_1} |w(s,y)|^{q^*}\, dy \ge M^{q^*-1} \int_{|y|\le k_1} |w(s,y)|\, dy \ge M^{q^*-1}\eps_0/2. $$
\end{proof}

\begin{proof}[Proof of Theorem~\ref{thm1}]
The proof is divided into five steps.  \\
{\bf Step 1.} {\it Preliminaries.}
We shall first prove the local assertion (ii). 
Let $a$ be a blow-up point of~$u$. We may take $T=1$ and $a=0$ without loss of generality.
As before, $u$ is extended by $0$ to the whole of $\R^N$ (we shall just write $u$ instead of $\tilde u$ without risk of confusion).
Also we shall denote $B_r=B_r(0)$.

Assume for contradiction that 
$$\liminf_{t\to 1}\int_{\Omega\cap B_{r_0}} |u(t,x)|^{q^*}dx<\infty$$
for some $r_0>0$.
Thus there exist $M_1>0$ and a sequence $t_n\to 1$ such that
\be\label{lowerA}
\|u(t_n)\|_{L^{q*}(B_{r_0})}\le M_1.
\ee 
Moreover, as a consequence of Lemma~\ref{lem31}, $u$ has finitely many blow-up points in the ball $B_{r_0}$. Therefore, there exists $r\in (0,r_0)$ 
such that $0$ is the only blow-up point in $B_r$ and, by standard parabolic estimates,
there exists $U\in L^{q*}(B_r)\cap L^\infty_{loc}(\overline B_r\setminus\{0\})$ such that
\be\label{cvtoU}
\lim_{t\to 1} u(t,x)=U(x), \quad\hbox{ for $0<|x|\le r$},
\ee
where the convergence is uniform on compact subsets of $\overline B_r\setminus\{0\}$.

\smallskip
{\bf Step 2.} {\it Rescaling and lower $L^{q^*}$estimate of the rescaled functions 
in some ball centered at $ 0 $.} We rescale $u$ as 
\be\label{defvn}
v_n(s,y)=(1-t_n)^\beta u\bigl(1+(1-t_n)s,y\sqrt{1-t_n}\bigr), \ \quad s\in(-2,0),\ y\in \R^N.
\ee 
The function $v_n$ satisfies
\be\label{eqnvn}
\partial_tv_n-\Delta v_n=|v_n|^{p-1}v_n,\quad s\in(-2,0),\ y\in \Omega_n:=(1-t_n)^{-1/2} (\Omega \cap B_r). 
\ee 
By the type~I estimate \eqref{typeI}, we have 
\be\label{estvn}
|v_n(s,y)|\le M|s|^{-\beta}, \ \quad s\in(-2,0),\ y\in \R^N.
\ee 
As $n\to \infty$, the domain $\Omega_n$ converges to $D:=\R^N$ if $0\in\Omega$ and to $D:=\R^N_+$ if $0\in\partial\Omega$
after suitable rotation if necessary. 
By parabolic estimates, it follows that, up to a subsequence, 
\be\label{cvlocvn}
\hbox{ $v_n$ converges locally uniformly in $(-2,0)\times \overline D$ to $v$},
\ee 
where $v$ is a classical solution of 
\be\label{eqnvD}
v_t-\Delta v=|v|^{p-1}v,\quad (s,y) \in (-2,0)\times D,
\ee
with $v=0$ on $(-2,0)\times \partial D$ in case $a=0\in\partial\Omega$.
In this second case, we extend $v$ by $0$ outside~$\overline D$.

Also, by Lemma~\ref{lem31}, since $0$ is a blow-up point of $u$, there exist $k_1,\eta_1>0$ such that
$$\begin{aligned}
\int_{|y|\le k_1} |v_n(s,y)|^{q^*}\, dy
&= (1-t_n)^{N/2}\int_{|y|\le k_1} |u(1+(1-t_n)s,y\sqrt{1-t_n})|^{q^*}\, dy\\ 
&=\int_{|x|\le k_1 \sqrt{1-t_n} } |u(1+(1-t_n)s,x)|^{q^*}\, dx\\
& \ge\int_{|x|\le k_1\sqrt{|s|(1-t_n)} } |u(1+(1-t_n)s,x)|^{q^*}\, dx \ge \eta_1,\quad  -1<s<0,
\end{aligned}$$
where we used $q^*=N(p-1)/2$. In view of \eqref{cvlocvn}, we deduce that 
\be\label{lowerbound}
\int_{|y|\le k_1} |v(s,y)|^{q^*}\, dy \ge \eta_1,\quad  -1<s<0.
\ee 

\smallskip
{\bf Step 3.} {\it Uniform a priori estimates of the rescaled functions $v_n$ for large $|y|.$}
We shall show that there exist $s_1\in (0,1)$ and $C_1, K_1>0$ such that, for all $K_2>K_1$ and all $n\ge n_1(K_2)$,
\be\label{convbetter}
|v_n(s,y)|\le C_1, \quad -s_1\le s<0,\ K_1\le |y|\le K_2.
\ee 
This step is more technical and requires some interplay between the different versions 
of the solution: $v_n$, $v$, $u$ and $w_b$, with a certain range of rescaling points $b$. 
In what follows, $ C $ denotes a generic constant varying from line to line.

\vskip 1mm
By \eqref{lowerA}, \eqref{defvn} and recalling $q^*=N(p-1)/2$, we have
$$\begin{aligned}
\int_{ |y|\le r/\sqrt{1-t_n} } |v_n(-1,y)|^{q^*}\, dy
&=(1-t_n)^{N/2}\int_{ |y|\le r/\sqrt{1-t_n}}|u(t_n,y\sqrt{1-t_n})|^{q^*}\, dy \\
&=\int_{B_{r}} |u(t_n,x)|^{q^*}\, dx\le M_1^{q^*},
\end{aligned}$$
hence
$v(-1,\cdot)\in L^{q*}(D)$ by Fatou's Lemma.
Moreover, we have $v(-1,\cdot)\in L^\infty(D)$
as a consequence of \eqref{estvn}. We may thus apply Proposition~\ref{prop21}(ii) to deduce the existence of $R>0$ such that
\be\label{boundC}
\sup_{(-1,0)\times (\overline D\cap \{|y|\ge R\})} |v|=:M_0<\infty.
\ee 
Note that, in view of \eqref{eqnvD} and \eqref{boundC}, by parabolic regularity, $v$ extends to a function
\be\label{extensionv}
v\in C^{1,2}\bigl([-1,0]\times (\overline D\cap \{|y|>R\})\bigr).
\ee
\vskip 1mm

On the other hand, by \eqref{cvlocvn} and 
\eqref{boundC}, for any $\eps\in(0,1)$ and $K>R$, there exists $n_0(\eps,K)$ such that for all $n\ge n_0(\eps,K)$, we have
$$|v_n(s,y)|\le M_0+1,\quad -1\le s\le -\eps, \ R\le |y|\le 3K,$$
hence in particular
$$(\eps(1-t_n))^\beta \bigl| u\bigl(1-\eps(1-t_n),y\sqrt{1-t_n}\bigr)\bigr |=\eps^\beta |v_n(-\eps,y)|\le (M_0+1)\eps^\beta,
\quad R\le |y|\le 3K.$$
For any $\xi\in \R^N$, this can be rewritten as
$$(\eps(1-t_n))^\beta \bigl| u\bigl(1-\eps(1-t_n),\xi\sqrt{1-t_n}+(y-\xi)\eps^{-1/2}\sqrt{\eps(1-t_n)}\bigr)\bigr |\le (M_0+1)\eps^\beta,
\quad R\le |y|\le 3K.$$
In terms of solutions rescaled by similarity variables (cf. section 2), this amounts to
$$\bigl| w_{\xi\sqrt{1-t_n}}\bigl(-\log(\eps(1-t_n)),(y-\xi)\eps^{-1/2}\bigr)\bigr |\le (M_0+1)\eps^\beta,
\qquad R\le |y|\le 3K.$$
In particular, for any $K\ge 2R$ and any $\xi$ such that $K\le |\xi|\le 2K$, we have,
for all $n\ge n_0(\eps,K)$, 
\be\label{boundxi}
\bigl| w_{\xi\sqrt{1-t_n}}\bigl(-\log(\eps(1-t_n)),z\bigr)\bigr| \le (M_0+1)\eps^\beta,
\qquad |z|\le K\eps^{-1/2}/2.
\ee
\vskip 1mm

Next take $q\ge p$ with $q>Np/2$  and let $\eps_0$ be as in Proposition~\ref{prop21}(i).
Recall that 
\be\label{boundMb}
\|w_b(\tau)\|_\infty\leq M\quad\hbox{ for any $\tau\in [s_0,\infty)$ and $b\in \R^N$,}
\ee
owing to \eqref{typeI}. From now on, we fix $\eps=\eps_1\in(0,1)$ and pick $R_1>2$, in such a way that 
\be\label{boundMb2}
(M_0+1)\eps_1^\beta \int_{\R^N} e^{-|z|^2/4}\,dz \le \frac{\eps_0}{2}
\quad\hbox{ and }\quad
M\int_{|z|> R_1} e^{-|z|^2/4}\,dz\le \frac{\eps_0}{2}.
\ee
Set $K_1:=2\max(R,\eps_1^{1/2}R_1)>4$ and take any $K\ge K_1$.
It follows from \eqref{boundxi}-\eqref{boundMb2} that,
for all $n\ge n_1(K):=n_0(\eps_1,K)$,
$$\begin{aligned}
\|w_{\xi\sqrt{1-t_n}}(-\log(\eps_1(1-t_n)),\cdot)\|_{L^1_\rho}&\le M\int_{|z|> R_1} e^{-|y|^2/4}\,dy \\
&\quad + (M_0+1)\eps_1^\beta \int_{|z|\le R_1} e^{-|y|^2/4}\,dy \le \eps_0,\qquad K\le |\xi|\le 2K.
\end{aligned}$$
By Proposition~\ref{prop21}(i) applied with $\sigma=\tau_n:=-\log(\eps_1(1-t_n))$, we deduce that 
\be\label{boundwxi}
\|w_{\xi\sqrt{1-t_n}}(\tau,\cdot)\|_{L^q_\rho}\le C \eps_0   e^{-\beta (\tau-\tau_n)},
\qquad \tau\ge \tau_n+s^*,\ \  K\le |\xi|\le 2K,\ \ n\ge n_1(K). 
\ee

Now set $t'_n:=1-\eps_1 e^{-s^*}(1-t_n)$. Restating \eqref{boundwxi} in terms of $u$, using 
$$u(t,b+z\sqrt{1-t}))=e^{\beta\tau}w_b(\tau,z), \quad\hbox{with $\tau=-\log(1-t)$,}$$ 
we thus have, for all $n\ge n_1(K)$, 
$$\int_{\R^N} \bigl|u(t,\xi\sqrt{1-t_n}+y\sqrt{1-t})\bigr|^q \rho(y)\, dy\le (C \eps_0)^q e^{\beta  q\tau_n}=C(1-t_n)^{-\beta q},
\ \quad t\in [t'_n,1),\ K\le |\xi|\le 2K.$$
Going back to $v_n$, noting that
$$v_n(s,\xi+y\sqrt{|s|})=(1-t_n)^\beta u\bigl(t,\xi\sqrt{1-t_n}+y\sqrt{1-t}\bigr),\quad\hbox{ with $t= 1+(1-t_n)s$,}$$
and that 
$$1+(1-t_n)s\ge t'_n=1-e^{-s^*}\eps_1(1-t_n)\Longleftrightarrow s\ge -e^{-s^*}\eps_1,$$
we obtain, for each $K\le |\xi|\le 2K$, 
$$\int_{|y|\le 1} \Bigl|v_n\bigl(s,\xi+y\sqrt{|s|}\bigr)\Bigr|^q \, dy\le e^{1/4}
\displaystyle\int_{\R^N} \Bigl|v_n\bigl(s,\xi+y\textstyle\sqrt{|s|}\bigr)\Bigr|^q \rho(y)\, dy\le C,
\qquad -\eps_1 e^{-s^*}\le s<0.$$
For any $K\ge K_1$ and each $\xi_0$ with $K+2\le |\xi_0|\le 2K-2$ (recalling $K_1>4$), 
we integrate with respect to $\xi\in B_2(\xi_0)$ and use the change of variable $z=\xi+y\sqrt{|s|}$. We thus get,
for all $n\ge n_1(K)$, 
$$\begin{aligned}
C 
&\ge \int_{|\xi-\xi_0|\le 2}\int_{|y|\le 1} \bigl|v_n\bigl(s,\xi+y\sqrt{|s|}\bigr)\bigr|^q \, dy\,d\xi \\
&= \int_{|y|\le 1} \int_{|z-y\sqrt{|s|}-\xi_0|\le 2}  \bigl|v_n(s,z)\bigr|^q \, dz\,dy
\ge \int_{|y|\le 1} \int_{|z-\xi_0|\le 1}  |v_n(s,z)|^q \, dz\,dy,
\end{aligned} $$
hence
$$\int_{|z-\xi_0|\le 1}  |v_n(s,z)|^q \, dz \le C,
\qquad -\eps_1 e^{-s^*}\le s<0. $$
By parabolic regularity applied to equation \eqref{eqnvn}, we deduce that, for all $K\ge  K_1$ and $n\ge n_1(K)$,
$$|v_n(s,y)|\le C,\quad\hbox{for all $s\in[-s_1,0)$ and $K+2\le |y|\le 2K-2$},$$
with $s_1:= (\eps_1/2)e^{-s^*}$, which implies the claim \eqref{convbetter}.

\medskip

{\bf Step 4.} {\it Backward uniqueness argument and completion of proof of assertion (ii).}
From \eqref{eqnvn} and \eqref{convbetter}, applying parabolic regularity again, we deduce that,
for any $K_2>K_1$, the limit
$$v_n(0,y)=\lim_{s\to 0} v_n(s,y) \quad\hbox{ exists for $K_1\le |y| \le K_2$ and $n\ge n_1(K_2)$}$$
and that
\be\label{cvvn}
\quad\hbox{ $v_n(0,y)\to v(0,y)$,\quad locally uniformly for $|y|\ge K_1$}
\ee 
(recalling \eqref{extensionv} and the fact that $v$ is extended by $0$ outside $D=\R^N_+$ if $a=0\in \partial\Omega$).
But, by \eqref{cvtoU} and \eqref{defvn}, taking $n_1(K_2)$ larger if necessary, we also have 
$$v_n(0,y)=(1-t_n)^\beta u\bigl(1,y\sqrt{1-t_n}\bigr)=(1-t_n)^\beta U(y\sqrt{1-t_n}),
\quad\hbox{ $K_1\le |y| \le K_2, \ \ n\ge n_1(K_2).$}$$ 
Since $U\in L^{q^*}(B_r)$, for each $y_0$ with $|y_0|\ge K_1+1$, we deduce that,
for all $n\ge n_1(|y_0|+1)$, 
$$
\begin{aligned}
\int_{B_1(y_0)} |v_n(0,y)|^{q^*}\,dy
&=\int_{B_1(y_0)} (1-t_n)^{N/2} |U(y\sqrt{1-t_n})|^{q^*}\, dy \\
&\le \int_{|z|\le (|y_0|+1)\sqrt{1-t_n}} |U(z)|^{q^*}\, dz \longrightarrow 0,\quad n\to\infty.
\end{aligned}
$$
Writing
$$\|v(0,\cdot)\|_{L^{q^*}(B_1(y_0))}
\le \|v(0,\cdot)-v_n(0,\cdot)\|_{L^{q^*}(B_1(y_0))}+ \|v_n(0,\cdot)\|_{L^{q^*}(B_1(y_0))}
$$
and since the first term also converges to $0$ by \eqref{cvvn}, we infer that
$$v(0,y)=0,\quad |y|\ge K_1.$$ 
By backward uniqueness (cf. Proposition~\ref{prop23backward-unique}), 
in view of \eqref{eqnvD}, \eqref{boundC}, \eqref{extensionv},
we deduce that
$$v(s,y)=0,\quad -1<s<0,\ |y|\ge K_1.$$ 
Finally, by unique continuation (e.g.~Theorem 4.1 of \cite{Escauriaza-Seregin-Sverak:UMN58}), 
it follows that actually $v\equiv 0$ on $(-1,0)\times D$.
We refer also to \cite{Escauriaza-Fernandez:AM41} and references therein for unique 
continuation theorem.
But this contradicts the lower bound in~\eqref{lowerbound}.
Assertion (ii) follows.

\smallskip
{\bf Step 5.} {\it Proof of assertion (i).}
We may assume that $u(t_0)\in L^{q*}(\Omega)$ for some $t_0\in (0,T)$
since otherwise the conclusion is immediate. 
 It follows from Propostion~\ref{prop21}(ii) that
$$|u(t,x)|\le C,\quad  t_0<t<T,\ x\in\overline\Omega\cap \{|x|\ge R\}$$
for some $C, R>0$.
Therefore, there exists at least a blow-up point $a\in\overline\Omega$.
The conclusion then follows from assertion (ii). 
\end{proof}

\section{Appendix: Type~II blow-up solutions with bounded critical $L^q$ norm}

Since the papers \cite{schweyer_jfa12} and \cite{dpmw18} did not  explicitly address the $L^{q^*}$ behavior of the solutions that they constructed,
we shall here check that they do satisfy property \eqref{noBUqcrit}.
\smallskip

\subsection{Case $p=p_S$ and $N=5$ (cf.~\cite{dpmw18})} 
Set
$$U(x)=c(N)(1+|x|^2)^{-(N-2)/2},$$ 
which is a regular positive solution of $\Delta u+u^p=0$ in $\R^N$ (the standard Aubin-Talenti bubble),
and denote its rescaling by $U_a(x)=a^{-2/(p-1)}U(a^{-1}x)$, $a>0$.
The type~II blow-up solution in \cite{dpmw18} (in the case of a single blow-up point) is of the form 
\be\label{formuU}
u(t,x)=U_{\lambda(t)}(x)+U^*(x)+\theta(t,x),\quad t\in (0,T),\ x\in \Omega,
\ee
where $\Omega$ is some smooth bounded domain of $\R^N$ with $0\in\Omega$.
Here $\lambda(t)\to 0$ as $t\to T$,
$U^* \in L^\infty(\Omega)$ and $\theta(t,\cdot) \to 0$ in $L^\infty(\Omega)$.
Since $q^*=N(p-1)/2=2N/(N-2)$ we have
$$\int_\Omega |U_{\lambda(t)}(x)|^{q^*}\, dx\le 
\lambda^{-N}(t)\int_{\R^N} |U(x/\lambda(t))|^{q^*}\, dx
=C\int_{\R^N} (1+|y|^2)^{-N}\, dy=C,
$$
where $ C $ denotes generic constants.
It follows that $u$ satisfies \eqref{noBUqcrit}.

\subsection{Case $p=p_S$ and $N=4$ (cf.~\cite{schweyer_jfa12})} 

The solution is also of the form~\eqref{formuU}, this time with $\Omega=\R^N$,
$\nabla U^* \in L^2(\R^N)$, $\nabla\theta(t,\cdot) \to 0$ in $L^2(\R^N)$, and $u_0 \in H^1(\R^N)$
(cf.~\cite[Theorem~1.1]{schweyer_jfa12}).
We note that $u(t,\cdot) \in H^1(\R^N)$ for each $t\in (0,T)$, due to the local well-posedness of \eqref{pbmP} in that space for $p=p_S$
(see e.g. \cite[Example~51.28]{QS07})). It then follows from the Sobolev inequality that
$$\begin{aligned}
\|u(t,\cdot)\|_{L^{q^*}(\R^N)} 
&\le \|U_{\lambda(t)}(\cdot)\|_{L^{q^*}(\R^N)}+\|u(t,\cdot)-U_{\lambda(t)}(\cdot)\|_{L^{q^*}(\R^N)} \\
&\le C+C\|\nabla u(t,\cdot)-\nabla U_{\lambda(t)}(\cdot)\|_{L^2(\R^N)} 
\le C\|\nabla U^*+\nabla\theta(t,\cdot)\|_{L^2(\R^N)} \le C,
\end{aligned}$$
and we conclude that $u$ also satisfies \eqref{noBUqcrit}.

\medskip

{\bf Acknowledgements.} 
NM was supported by the JSPS Grant-in-Aid for Scientific Research (B) (No.26287021).

\end{document}